\documentclass[11pt]{article}
\usepackage[T1]{fontenc}
\usepackage[utf8]{inputenc}
\usepackage{amsmath,amssymb,amsthm}
\usepackage{geometry}
\newcommand{\E}{\mathbb E}
\newcommand{\Bin}{\operatorname{Bin}}
\newcommand{\Z}{\mathbb Z}
\newcommand{\R}{\mathbb R}
\newcommand{\Mad}{\operatorname{MAD}}

\theoremstyle{plain}
\newtheorem{theorem}{Theorem}[section]
\newtheorem{lemma}[theorem]{Lemma}
\newtheorem{proposition}[theorem]{Proposition}
\newtheorem{corollary}[theorem]{Corollary}
\theoremstyle{remark}
\newtheorem{remark}[theorem]{Remark}
\theoremstyle{definition}
\newtheorem{example}[theorem]{Example}

\title{Local binomial expansions with an Appell shift,\\
and the mean absolute deviation of the binomial distribution}
\author{
N.~Elezovi\'c\\
Department of Applied Mathematics,\\
Faculty of Electrical Engineering and Computing,\\
University of Zagreb, 10000 Zagreb, Croatia\\
\texttt{neven.elezovic@fer.hr}
}
\date{\today}

\begin{document}
\maketitle

\begin{abstract}
	We derive complete asymptotic expansions for the binomial mass at a bounded
	lattice displacement and for the mean absolute deviation
	\(\E|X-Np|\), \(X\sim\Bin(N,p)\), with \(0<p<1\).  De~Moivre's exact formula
	reduces the latter problem to the local mass at
	\(\nu=\lceil Np\rceil\), so the coefficients depend on the oscillating
	displacement \(h_N=\lceil Np\rceil-Np\).  We show that the full expansion is
	governed by Bernoulli polynomials evaluated at this displacement; equivalently,
	the lattice correction is an Appell shift in the Stirling series.  The
	calculation is based on a gamma-quotient expansion with unequal linear
	scalings, stated with uniformity in the bounded shift.  In passing from the local
	mass to the mean absolute deviation, the elementary, non-Bernoulli part of the
	shift cancels term by term against the De~Moivre prefactor, leaving coefficients
	that are pure Bernoulli polynomials.  As consequences, the classical first
	correction of Frame and Johnson is embedded in the general coefficient sequence,
	and the Ces\`aro means of the oscillating coefficients are obtained from the
	multiplication theorem for Bernoulli polynomials.  Finally, although an oscillating
	asymptotic expansion does not bound its function by truncation, De~Moivre's identity
	together with Robbins's form of Stirling's formula yields an elementary two-sided
	bound of logarithmic width \(O(N^{-2})\) in the interior; and at an integer mean the
	logarithmic expansion reduces to a sign-alternating series in odd powers of \(N^{-1}\) which
	we prove to be enveloping, via a sign-definite Binet-kernel representation of the combined
	three-gamma Stirling remainder: successive truncations bracket the mean absolute deviation.
\end{abstract}



% =====================================================================
\section{Introduction}

	Let $X\sim\Bin(N,p)$ with $0<p<1$ and $q=1-p$.  The quantity
	\begin{equation}\label{eq:mad-def}
		\Mad(N,p):=\E\,|X-Np|
	\end{equation}
	is the mean absolute deviation of $X$ about its mean.  It is one of the oldest explicitly
	summed quantities in probability: De~Moivre showed that the sum collapses to a single term,
	\begin{equation}\label{eq:demoivre}
		\E\,|X-Np| \;=\; 2\,\nu\, q \binom N\nu p^{\nu}q^{N-\nu},
		\qquad \nu:=\lceil Np\rceil .
	\end{equation}
	(For the history of \eqref{eq:demoivre} and of the family of identities it belongs to, see
	Diaconis and Zabell~\cite{diaconis_zabell}; for the mean absolute deviation as a descriptive
	functional of the binomial law, see Johnson, Kemp and Kotz~\cite[\S3.3]{jkk}.)  A short proof is
	recalled in Section~\ref{sec:mad}.
	
	Formula \eqref{eq:demoivre} is exact at every finite $N$, and it is the natural starting point
	for asymptotics: the whole expectation is carried by the local mass of the binomial
	distribution at one lattice point.  The leading behaviour is classical and is the Gaussian mean
	deviation,
	\begin{equation}\label{eq:mad-leading}
		\E\,|X-Np|\;\sim\;\sqrt{\frac{2Npq}{\pi}} .
	\end{equation}
	
	Beyond \eqref{eq:mad-leading} the literature stops at the \emph{first} correction, and then turns
	to bounds.  The first correction is due to Frame~\cite{frame1945}, who states it in the squared
	form
	\begin{equation}\label{eq:frame}
		\frac\pi2\bigl(\E|X-Np|\bigr)^2
		=Npq+\bigl(Np-[Np]\bigr)\bigl(Nq-[Nq]\bigr)-\frac{1-pq}{6}+\frac{E_N}{24N},
		\qquad |E_N|\le1
	\end{equation}
	(asymptotically), the next order being merely bounded and not computed; Johnson~\cite{johnson1957}
	obtains the same correction in linearised form, and remarks that it already gives three-figure
	accuracy for $0.2\le p\le0.8$.  Both authors have the lattice defect in it: for
	$Np\notin\Z$ the product $(Np-[Np])(Nq-[Nq])$ is exactly $h_N(1-h_N)$, with $h_N$ as in
	\eqref{eq:hN} below.  This is more than an anecdote, and we return to it in
	Remark~\ref{rem:frame}: written as a Bernoulli polynomial, $h(1-h)=\tfrac16-B_2(h)$, and
	\eqref{eq:frame} \emph{is} the first coefficient of Theorem~\ref{thm:mad}.  The Bernoulli pattern
	was therefore visible, at one order, in 1945; what was missing was the mechanism that continues it.
	Frisch~\cite{frisch1924} and Jordan~\cite{jordan1927} carry only the exact formula
	\eqref{eq:demoivre}, and Ramasubban~\cite{ramasubban1958}, having noted that Johnson had settled
	the binomial, treats the other discrete laws.
	
	After that the subject becomes one of \emph{bounds} rather than expansions.
	Blyth~\cite{blyth1980} evaluates $D_N(p)=\E|X/N-p|$ exactly, by \eqref{eq:demoivre}, and extracts
	its leading asymptotics --- for instance $\sqrt N\max_pD_N(p)\to(2\pi)^{-1/2}$, by Stirling's
	formula.  Berend and Kontorovich~\cite{berend_kontorovich} give sharp \emph{non-asymptotic}
	bounds on $\E|X-Np|$, valid for all $p\in[0,1]$, and emphasise that the approximation of the mean
	absolute deviation by a multiple of the standard deviation is \emph{not uniform} in $p$ and
	degenerates at the endpoints.  We have not found in the literature any correction beyond the first,
	let alone a full expansion.  Theorem~\ref{thm:mad} supplies one, to all orders; and the
	non-uniformity of \cite{berend_kontorovich} is legible in it, in the factors $(pq)^{-m}$ carried by
	the $m$-th coefficient.
	
	What is not classical is what comes next.  The obstruction is visible in
	\eqref{eq:demoivre}: the mean $Np$ is in general \emph{not} a lattice point, and the formula
	evaluates the mass at the first lattice point above it.  Write
	\begin{equation}\label{eq:hN}
		h_N:=\nu-Np=\lceil Np\rceil-Np\in[0,1).
	\end{equation}
	The displacement $h_N$ is bounded but does not converge; for irrational $p$ it is equidistributed
	in $[0,1)$.  Consequently the correction terms in the expansion of \eqref{eq:demoivre} cannot be
	constants: they must be bounded functions of $h_N$, and the expansion \emph{oscillates}.  This is
	the same phenomenon that produces the Bernoulli-polynomial terms in Edgeworth expansions for
	lattice variables (Esseen~\cite{esseen1945}, Kolassa and McCullagh~\cite{kolassa_mccullagh1990}),
	here in an exactly solvable instance.
	
	The purpose of this note is to compute the expansion completely, and to identify the algebraic
	mechanism behind it.  Our main results are the following two.
	
\begin{itemize}
\item[(A)] \emph{The local binomial mass with a bounded displacement.}  For $h$ ranging over a
	fixed compact interval and $Np+h\in\Z$, the mass
	$\pi_N(h;p)=\binom{N}{Np+h}p^{Np+h}q^{Nq-h}$ has a complete logarithmic
	expansion whose coefficients are built from Bernoulli polynomials
	$B_{k+1}(h)$; see Theorem~\ref{thm:local}.
\item[(B)] \emph{The mean absolute deviation.}  De~Moivre's formula then gives
	a complete expansion of \(\Mad(N,p)\).  Its coefficients are obtained from
	the same Appell-shifted series at the oscillating displacement \(h_N\), after
	the elementary prefactor in \eqref{eq:demoivre} is absorbed; see
	Theorem~\ref{thm:mad}.
\end{itemize}

\medskip\noindent
\textbf{Convention on asymptotic statements.}  All expansions below are \emph{complete}, in the
	following sense.  Writing $f(x)\sim g(x)\sum_{n\ge0}c_nx^{-n}$ means: for every fixed $M$,
	\[
		f(x)=g(x)\Bigl(\sum_{n=0}^{M}c_nx^{-n}+O(x^{-M-1})\Bigr),
	\]
	with the implied constant uniform in all parameters ranging over the compact sets specified in
	the statement.  When an answer is displayed as the exponential of a series, the logarithmic
	series is truncated at a fixed order and the exponential re-expanded to the same order;
	Theorem~\ref{thm:multiplicative} shows that the two forms carry
	the same information.  The word \emph{complete} refers to this control of every fixed
	truncation order, not to convergence of the series: here the coefficients $d_m(h_N;p)$ do
	not even settle down as $N\to\infty$, since they depend on the oscillating $h_N$.
	Uniformity is not a technicality here.  The displacement $h_N$ has no
	limit, so a statement about coefficients that depend on $h_N$ is empty unless the error term is
	uniform in the displacement; this is why the compactness hypotheses are carried through every
	statement below.

\medskip
	The technical device is stated in Section~\ref{sec:unequal}.  In the theory of asymptotic
	expansions of gamma quotients developed in \cite{be1,be3,be4} and applied to binomial
	coefficients in \cite{be-binom}, the gamma factors all carry the \emph{same} large argument
	$x$, and the coefficients are then obtained from a single recursion in the Bernoulli
	polynomials evaluated at the shifts.  Binomial coefficients of the form $\binom{px}{rx}$ are
	brought into that form by the multiplication formula for the gamma function --- a step which
	requires the scalings to be commensurable, and which therefore forces a case distinction
	according to the arithmetic of $p$ and $r$.  Here the three gamma factors carry the arguments
	$N$, $Np$ and $Nq$ with an arbitrary real $p$, so no such equalisation is available.
	Lemma~\ref{lem:unequal} shows that it is not needed: the recursion of \cite{be-binom} survives
	verbatim provided each Bernoulli polynomial is weighted by the $(-n)$-th power of its own
	scaling factor.  This is the natural home of the computation, and it is what makes the Appell
	shift (in the sense of \cite{bev-appell}) visible.  The lemma itself is not new --- its history
	runs from Box~\cite{box1949} through Gupta and Tang~\cite{gupta_tang} to Karp and
	Prilepkina~\cite{karp_prilepkina}, and, for unit scalings, through \cite{be1,be4} to Theorem A of
	\cite{be-binom}.  What we require of it, and what the sources do not state, is the unequal real
	scaling and uniformity in the shift.

% =====================================================================
\section{Gamma quotients with unequal scalings}\label{sec:unequal}

	Throughout, $B_k$ and $B_k(t)$ denote the Bernoulli numbers and polynomials in the standard
	normalisation,
	\[
		\frac{z}{e^{z}-1}=\sum_{k\ge0}B_k\,\frac{z^k}{k!},
		\qquad
		\frac{z\,e^{tz}}{e^{z}-1}=\sum_{k\ge0}B_k(t)\,\frac{z^k}{k!},
	\]
	so that $B_k=B_k(0)$ and, in particular, $B_1=-\tfrac12$.  The sequence $(B_k(t))_{k\ge0}$ is the
	Appell sequence generated by $z/(e^z-1)$; for Appell sequences we refer to Roman~\cite{roman}.
	We shall use the Stirling expansion in its shifted (Appell) form
	\begin{equation}\label{eq:stirling-shift}
		\log\Gamma(x+t)\sim\Bigl(x+t-\tfrac12\Bigr)\log x-x+\tfrac12\log 2\pi
		+\sum_{n\ge1}\frac{(-1)^{n+1}B_{n+1}(t)}{n(n+1)}\,\frac1{x^n},
		\qquad x\to\infty,
	\end{equation}
	uniformly for $t$ in compact sets; see \cite[\S4]{bev-appell} and \cite{luke}.  Formula
	\eqref{eq:stirling-shift} is the log-gamma instance of the Appell principle: at $t=0$ the
	coefficients are the Bernoulli numbers, and shifting the argument by $t$ replaces them by the
	Bernoulli polynomials $B_{n+1}(t)$, which form an Appell sequence.  The general statement --- any
	asymptotic expansion in inverse powers acquires Appell polynomials when the variable is shifted
	--- is \cite[Theorem 3.1]{bev-appell}.

\begin{lemma}[Gamma quotients with unequal scalings, \cite{karp_prilepkina, be4, be-binom}]\label{lem:unequal}
	Let $\rho,\sigma\ge1$, let $\lambda_1,\dots,\lambda_\rho>0$ and $\mu_1,\dots,\mu_\sigma>0$,
	and let $u_1,\dots,u_\rho$, $v_1,\dots,v_\sigma$ be real.  Put
	\[
		F(x)=\frac{\prod_{j=1}^{\rho}\Gamma(\lambda_j x+u_j)}
		          {\prod_{k=1}^{\sigma}\Gamma(\mu_k x+v_k)} .
	\]
	Then, as $x\to\infty$,
	\begin{equation}\label{eq:logF}
		\log F(x)\sim \Lambda\,x\log x+\Theta\,x+U\log x+K
		+\sum_{n\ge1}\frac{(-1)^{n+1}}{n(n+1)}\,S_{n+1}\,\frac{1}{x^{n}},
	\end{equation}
	where
	\begin{align}
		\Lambda&=\sum_j\lambda_j-\sum_k\mu_k, &
		\Theta&=\sum_j\lambda_j\log\lambda_j-\sum_k\mu_k\log\mu_k-\Lambda,
		\label{eq:LambdaTheta}\\
		U&=\sum_j\Bigl(u_j-\tfrac12\Bigr)-\sum_k\Bigl(v_k-\tfrac12\Bigr), &
		K&=\sum_j\Bigl(u_j-\tfrac12\Bigr)\log\lambda_j
		 -\sum_k\Bigl(v_k-\tfrac12\Bigr)\log\mu_k+\frac{\rho-\sigma}{2}\log2\pi,
		\label{eq:UK}
	\end{align}
	and
	\begin{equation}\label{eq:Sn}
		S_{n+1}=\sum_{j=1}^{\rho}\lambda_j^{-n}B_{n+1}(u_j)
		       -\sum_{k=1}^{\sigma}\mu_k^{-n}B_{n+1}(v_k).
	\end{equation}
	More precisely, fix a compact $H\subset\R$ and constants $0<c\le C<\infty$.  Then for every
	$M\ge0$ there are $x_M$ and $C_M$, depending only on $M,H,c,C,\rho,\sigma$, such that for all
	$x\ge x_M$, all shifts $u_j,v_k\in H$ and all scalings $\lambda_j,\mu_k\in[c,C]$,
	\begin{equation}\label{eq:logF-trunc}
		\Bigl|\,
		\log F(x)-\Lambda x\log x-\Theta x-U\log x-K
		-\sum_{n=1}^{M}\frac{(-1)^{n+1}}{n(n+1)}\,S_{n+1}\,\frac1{x^{n}}
		\,\Bigr|\;\le\;\frac{C_M}{x^{M+1}} .
	\end{equation}
\end{lemma}

\begin{proof}
	Apply \eqref{eq:stirling-shift} to each factor with $x$ replaced by $\lambda_jx$
	(respectively $\mu_kx$) and $t$ by $u_j$ (respectively $v_k$), and use
	$\log(\lambda_jx)=\log\lambda_j+\log x$:
	\[
	\begin{aligned}
		\log\Gamma(\lambda_jx+u_j)\sim{}&
		\lambda_j\,x\log x+\lambda_j(\log\lambda_j-1)\,x
		+\Bigl(u_j-\tfrac12\Bigr)\log x\\
		&{}
		+\Bigl(u_j-\tfrac12\Bigr)\log\lambda_j+\tfrac12\log2\pi
		+\sum_{n\ge1}
		\frac{(-1)^{n+1}B_{n+1}(u_j)}{n(n+1)\lambda_j^{n}}\,\frac1{x^n}.
	\end{aligned}
	\]
	Collecting the terms of each type over the numerator and the denominator gives
	\eqref{eq:logF}--\eqref{eq:Sn}.

	For \eqref{eq:logF-trunc}: since $\lambda_j\ge c$ and the shifts lie in the compact $H$, all
	arguments $\lambda_jx+u_j$ and $\mu_kx+v_k$ are positive once $x\ge x_0(H,c)$, so every
	logarithm above is real; and the large variable of each factor satisfies
	$\lambda_jx\ge cx\to\infty$ uniformly.  The truncated form of \eqref{eq:stirling-shift}, with a
	remainder uniform in the shift, needs no separate source.  Apply the \emph{unshifted} Stirling
	expansion --- whose remainder after $M$ terms is $O(w^{-M-1})$ as $w\to\infty$, by the classical
	error bounds \cite[\S5]{luke},\cite[\S2.1]{paris_kaminski} --- to $\log\Gamma(w)$ at
	$w=\lambda_jx+u_j$; since $\lambda_j\ge c$ and $u_j\in H$ we have $w\ge cx-\sup_H|u|\to\infty$
	uniformly, so that remainder is $O(x^{-M-1})$ uniformly.  It remains to re-expand the elementary
	factors $\log w=\log(\lambda_jx)+\log\bigl(1+u_j/(\lambda_jx)\bigr)$ and
	$w^{1-2n}=(\lambda_jx)^{1-2n}\bigl(1+u_j/(\lambda_jx)\bigr)^{1-2n}$ in powers of $x^{-1}$; for
	$x\ge x_0(H,c)$ one has $|u_j/(\lambda_jx)|\le\tfrac12$, so each is a Taylor expansion with a
	remainder $O(x^{-M-1})$ that is uniform in $u_j\in H$ and $\lambda_j\ge c$.  Collecting the powers
	of $x^{-n}$ carries the Bernoulli numbers into the Bernoulli polynomials $B_{n+1}(u_j)$ --- this is
	the Appell shift \cite[Thm 3.1]{bev-appell} --- and yields \eqref{eq:stirling-shift} in truncated
	form.  The coefficients $S_{n+1}$ of \eqref{eq:Sn} are then bounded on the parameter set, being
	polynomials in the $u_j,v_k$ and in $\lambda_j^{-1},\mu_k^{-1}\le c^{-1}$.  Summing the $\rho+\sigma$
	truncated expansions gives \eqref{eq:logF-trunc}.
\end{proof}

\begin{theorem}[Multiplicative form]\label{thm:multiplicative}
	In the notation of Lemma~\ref{lem:unequal}, suppose $\Lambda=0$.  Then
	\begin{equation}\label{eq:Fmult}
		F(x)\sim e^{K}\,e^{\Theta x}\,x^{U}\sum_{n\ge0}P_n\,x^{-n},
	\end{equation}
	where $P_0=1$ and
	\begin{equation}\label{eq:Prec}
		P_n=\frac1n\sum_{k=1}^{n}\frac{(-1)^{k+1}}{k+1}\,S_{k+1}\,P_{n-k},
		\qquad n\ge1 .
	\end{equation}
	Under the compactness hypotheses of Lemma~\ref{lem:unequal} this holds in the truncated form:
	for every $M$,
	\begin{equation}\label{eq:Fmult-trunc}
		F(x)=e^{K}\,e^{\Theta x}\,x^{U}
		\Bigl(\sum_{n=0}^{M}P_nx^{-n}+O(x^{-M-1})\Bigr),
	\end{equation}
	with the implied constant uniform on the parameter set.
\end{theorem}

\begin{proof}
	By \eqref{eq:logF} with $\Lambda=0$ we have
	\[
		F(x)\sim e^{K}e^{\Theta x}x^{U}
		\exp\Bigl(\sum_{n\ge1}A_nx^{-n}\Bigr),
		\qquad
		A_n=\frac{(-1)^{n+1}S_{n+1}}{n(n+1)} .
	\]
	If $\exp(\sum_{n\ge1}A_ny^{n})=\sum_{n\ge0}P_ny^n$, then
	differentiating and comparing coefficients gives
	$nP_n=\sum_{k=1}^{n}kA_kP_{n-k}$, which is \eqref{eq:Prec}.

	For \eqref{eq:Fmult-trunc}, exponentiate \eqref{eq:logF-trunc}.  The coefficients $A_n$ are
	bounded on the parameter set, so
	\[
		\exp\bigl(\sum_{n\le M}A_nx^{-n}\bigr)=\sum_{n\le M}P_nx^{-n}+O(x^{-M-1})
	\] 
	uniformly, while
	the remainder contributes the factor $\exp\bigl(O(x^{-M-1})\bigr)=1+O(x^{-M-1})$.  Multiplying
	the two gives an error of the same order.
\end{proof}

% =====================================================================
\section{The local binomial mass}\label{sec:local}

Fix $0<p<1$, $q=1-p$.  For a real displacement $h$ such that $Np+h\in\{0,1,\dots,N\}$, define the
local binomial mass
\begin{equation}\label{eq:pi-def}
	\pi_N(h;p):=\binom{N}{Np+h}\,p^{Np+h}\,q^{Nq-h}
	=\Pr\{X=Np+h\},\qquad X\sim\Bin(N,p).
\end{equation}
The probabilities are kept inside: only then does the exponential (entropy) factor cancel and a
genuine local expansion appear.

\begin{theorem}[Local binomial mass]\label{thm:local}
	Fix $0<p<1$ and let $H\subset\R$ be compact.  For $h\in H$ such that
	$Np+h\in\{0,1,\dots,N\}$, as $N\to\infty$,
	\begin{equation}\label{eq:local-log}
		\pi_N(h;p)\sim\frac1{\sqrt{2\pi Npq}}
		\exp\Bigl(\sum_{k\ge1}\frac{A_k(h;p)}{N^k}\Bigr),
	\end{equation}
	where
	\begin{equation}\label{eq:Ak}
		A_k(h;p)=\frac{(-1)^{k+1}}{k(k+1)}
		\Bigl[B_{k+1}-\bigl(p^{-k}+(-1)^{k+1}q^{-k}\bigr)B_{k+1}(h)\Bigr]
		+\frac{(-1)^{k}h^{k}}{k\,p^{k}} .
	\end{equation}
	Equivalently, in multiplicative form,
	\begin{equation}\label{eq:local-mult}
		\pi_N(h;p)\sim\frac1{\sqrt{2\pi Npq}}\sum_{m\ge0}\frac{c_m(h;p)}{N^m},
		\qquad c_0=1,\quad
		c_m=\frac1m\sum_{k=1}^{m}k\,A_k(h;p)\,c_{m-k}.
	\end{equation}
	For every $M$ the remainder after $M$ terms is $O(N^{-M-1})$ relative to the leading factor,
	uniformly for such $h\in H$, with an implied constant depending on $M$, $p$ and $H$.  If
	$p$ is restricted to a compact subinterval of $(0,1)$, the constants may be chosen uniformly in
	$p$.
\end{theorem}

\begin{proof}
	Write the mass as a gamma quotient,
	\[
		\pi_N(h;p)=\frac{\Gamma(N+1)}{\Gamma(Np+h+1)\,\Gamma(Nq-h+1)}\;p^{Np+h}q^{Nq-h},
	\]
	and read the three factors in the format of Lemma~\ref{lem:unequal} with $x=N$:
	\[
		\Gamma(N+1)=\Gamma(1\cdot N+1),\qquad
		\Gamma(Np+h+1)=\Gamma\bigl(p\cdot N+(h+1)\bigr),\qquad
		\Gamma(Nq-h+1)=\Gamma\bigl(q\cdot N+(1-h)\bigr).
	\]
	Thus $\rho=1$, $\sigma=2$, and
	\[
		\lambda_1=1,\ u_1=1;\qquad \mu_1=p,\ v_1=h+1;\qquad \mu_2=q,\ v_2=1-h .
	\]
	Since $p+q=1$ we have $\Lambda=1-p-q=0$, so Theorem~\ref{thm:multiplicative} applies.  The
	elementary constants are
	\[
		\Theta=-\,(p\log p+q\log q),\qquad
		U=\tfrac12-\bigl(h+\tfrac12\bigr)-\bigl(\tfrac12-h\bigr)=-\tfrac12,
	\]
	\[
		K=-\Bigl(h+\tfrac12\Bigr)\log p-\Bigl(\tfrac12-h\Bigr)\log q-\tfrac12\log2\pi .
	\]
	Now multiply by $p^{Np+h}q^{Nq-h}$.  The factor $e^{\Theta N}$ cancels $p^{Np}q^{Nq}$
	identically, and
	\[
		e^{K}\,p^{h}q^{-h}
		=p^{-h-\frac12}q^{h-\frac12}(2\pi)^{-\frac12}\cdot p^{h}q^{-h}
		=\frac{1}{\sqrt{2\pi pq}},
	\]
	so that together with $x^{U}=N^{-1/2}$ we obtain the prefactor $(2\pi Npq)^{-1/2}$.  The
	series part is \eqref{eq:logF}: with $S_{k+1}=B_{k+1}(1)-p^{-k}B_{k+1}(h+1)-q^{-k}B_{k+1}(1-h)$,
	the coefficient of $N^{-k}$ in the logarithm is
	\[
		A_k=\frac{(-1)^{k+1}}{k(k+1)}
		\Bigl[B_{k+1}(1)-p^{-k}B_{k+1}(h+1)-q^{-k}B_{k+1}(1-h)\Bigr].
	\]
	For $k\ge1$ we have $k+1\ge2$, hence
	$B_{k+1}(1)=B_{k+1}$.  The Appell difference identity gives
	$B_{k+1}(h+1)=B_{k+1}(h)+(k+1)h^{k}$, and the reflection identity gives
	$B_{k+1}(1-h)=(-1)^{k+1}B_{k+1}(h)$.  Therefore
	\[
		S_{k+1}=B_{k+1}-\bigl(p^{-k}+(-1)^{k+1}q^{-k}\bigr)B_{k+1}(h)-(k+1)\,h^{k}p^{-k},
	\]
	and multiplying by $(-1)^{k+1}/(k(k+1))$ yields \eqref{eq:Ak}; the last term is
	\[
		\frac{(-1)^{k+1}\bigl(-(k+1)h^kp^{-k}\bigr)}{k(k+1)}
		=\frac{(-1)^k h^k}{k\,p^k}.
	\]
	This proves \eqref{eq:local-log}, and \eqref{eq:local-mult} follows as in
	Theorem~\ref{thm:multiplicative}.
\end{proof}

\begin{remark}
	The uniformity is essential.  For a fixed pair $(p,h)$ the set of integers
	$N$ with $Np+h\in\Z$ may be finite or empty --- for irrational $p$ it typically is --- so the
	statement has content only as a uniform assertion in $h$.  This is why the
	theorem is stated for $h$ in a compact set, and it is what permits the
	substitution of an $N$-dependent displacement such as $h_N$ in Section~\ref{sec:mad}.
\end{remark}

\begin{remark}[Where the Bernoulli polynomials come from, and why they stay]\label{rem:stay}
	The two identities used in the last step are precisely the two structural properties of an
	Appell sequence: the difference identity encodes $B_m'=mB_{m-1}$, and the reflection identity
	encodes the symmetry of the generating function $x/(e^x-1)$ (see
	\cite[Lemma 2.3]{bev-appell}).  They reduce the three shifts $1$, $h+1$, $1-h$ to a single
	Bernoulli polynomial $B_{k+1}(h)$ together with an elementary power $h^k$, but they cannot
	remove $B_{k+1}(h)$ itself: for a general real $h$ there is no further identity available.
	The displacement is therefore an irreducible parameter of the answer.
\end{remark}

\begin{remark}[Relation with the local limit theorems]\label{rem:llt}
	The local mass of the binomial has of course been expanded before, and it is worth saying
	exactly what Theorem~\ref{thm:local} adds.  Prokhorov~\cite{prokhorov1953} gives the local
	limit theorem for the binomial with an $O(N^{-1})$ error; the general theory of
	\emph{complete} local expansions for sums of lattice vectors is due to Esseen~\cite{esseen1945}
	and, in its modern form, to Bhattacharya and Ranga Rao~\cite[Thm 22.1]{bhattacharya_rao} (see
	also Petrov~\cite[Ch.~VII]{petrov}).  Closest to the present statement is
	Ouimet~\cite{ouimet2021}, whose Theorem 2.1 expands the multinomial mass --- the binomial being
	the case $d=1$ --- explicitly through order $N^{-1}$, uniformly in the bulk, and by the same
	elementary route: Stirling's formula, followed by Taylor expansion.  That $O(N^{-1})$
	multinomial expansion is itself older, going back to Arenbaev~\cite{arenbaev1976} and, in
	symmetrised form, to Siotani and Fujikoshi~\cite{siotani_fujikoshi1984}, both obtained by
	Fourier inversion; Ouimet reproduces it by the elementary route.

	These results are not in competition with Theorem~\ref{thm:local}; they are in a different
	gauge.  The classical lattice expansions are organised in powers of $N^{-1/2}$, with
	coefficients polynomial in the standardised point $x=(k-Np)/\sqrt{Npq}$.  Setting $x=h/\sqrt{Npq}$
	with $h$ bounded, and re-collecting, will indeed produce an expansion in integer powers of
	$N^{-1}$ with coefficients polynomial in $h$; in that sense Theorem~\ref{thm:local} is
	\emph{derivable} from the classical machinery, and we do not claim otherwise.  What that route
	does not produce is the \emph{closed form} of the coefficients.  The polynomials in $h$ come out
	of it one order at a time, as the by-product of a resummation, whereas here they are Bernoulli
	polynomials with the explicit weights $p^{-k}+(-1)^{k+1}q^{-k}$, generated by a single
	recursion, to all orders at once, and by exact algebra rather than Fourier inversion.  That is
	the whole content of the theorem, and it is what makes Theorem~\ref{thm:mad} possible.
\end{remark}

The first two coefficients are, explicitly,
\[
	A_1(h;p)=\frac1{12}-\frac{B_2(h)}{2pq}-\frac hp,
	\qquad
	A_2(h;p)=\frac{q-p}{6p^2q^2}\,B_3(h)+\frac{h^2}{2p^2},
\]
where we used $p^{-1}+q^{-1}=(pq)^{-1}$ and $p^{-2}-q^{-2}=(q-p)(pq)^{-2}$.

% =====================================================================
\section{The mean absolute deviation}\label{sec:mad}

We first recall De~Moivre's identity, in the short form we need.

\begin{proposition}[De~Moivre]\label{prop:demoivre}
	Let $X\sim\Bin(N,p)$, $q=1-p$, $\nu=\lceil Np\rceil$.  Then \eqref{eq:demoivre} holds:
	$\E|X-Np|=2\nu q\,b(\nu;N,p)$, where $b(k;N,p)=\binom Nkp^kq^{N-k}$.  When $Np\in\Z$ the
	choices $\nu=Np$ and $\nu=Np+1$ give the same value.
\end{proposition}

\begin{proof}
	Since $\E(X-Np)=0$ we have $\E|X-Np|=2\,\E(Np-X)^{+}$.  With $m=\lfloor Np\rfloor$ and
	$B(\ell;\kappa,p)=\Pr\{\Bin(\kappa,p)\le\ell\}$, the stop-loss computation
	\[
		\E(Np-X)^{+}=\sum_{k=0}^{m}(Np-k)b(k;N,p)
		=Np\,B(m;N,p)-Np\,B(m-1;N-1,p)
	\]
	uses $k\binom Nk p^kq^{N-k}=Np\,b(k-1;N-1,p)$.  The one-trial recursion
	$B(m;N,p)=q\,B(m;N-1,p)+p\,B(m-1;N-1,p)$ turns the bracket into $q\,b(m;N-1,p)$, so
	$\E|X-Np|=2Npq\,b(m;N-1,p)$.  Finally $Np\,b(m;N-1,p)=\nu\,b(\nu;N,p)$: for $Np\notin\Z$ we
	have $\nu=m+1$ and this is the absorption identity $\nu\binom N\nu=N\binom{N-1}{\nu-1}$ read
	backwards, while for
	$Np=m\in\Z$ both sides equal $Np\,b(m;N-1,p)$, and then
	$b(m+1;N,p)/b(m;N,p)=\frac{(N-m)p}{(m+1)q}=\frac{Np}{Np+1}$ shows that
	$\nu=m$ and $\nu=m+1$ give the same product $\nu\,b(\nu;N,p)$.
\end{proof}

In the notation of Section~\ref{sec:local}, \eqref{eq:demoivre} says
\begin{equation}\label{eq:mad-as-local}
	\E\,|X-Np| = 2\,(Np+h_N)\,q\;\pi_N(h_N;p),
	\qquad h_N=\lceil Np\rceil-Np .
\end{equation}
The prefactor $2(Np+h_N)q=2Npq\,(1+h_N/(Np))$ contributes
$2Npq\cdot(2\pi Npq)^{-1/2}=\sqrt{2Npq/\pi}$ to the leading order --- which is
\eqref{eq:mad-leading} --- and, in the logarithm, the series
\begin{equation}\label{eq:log-prefactor}
	\log\Bigl(1+\frac{h}{Np}\Bigr)=\sum_{k\ge1}\frac{(-1)^{k+1}}{k}\,\frac{h^k}{p^kN^k}.
\end{equation}
Comparing \eqref{eq:log-prefactor} with the last term of \eqref{eq:Ak}, we see that the two
cancel term by term, for every $k$.  This is the structural point of the note, and it deserves
to be stated as such.

\begin{theorem}[Mean absolute deviation of the binomial distribution]\label{thm:mad}
	Let $X\sim\Bin(N,p)$, $0<p<1$, $q=1-p$, and $h_N=\lceil Np\rceil-Np\in[0,1)$, with the
	convention $h_N=0$ when $Np\in\Z$ (by Proposition~\ref{prop:demoivre} the alternative choice
	$\nu=Np+1$, that is $h=1$, gives the same value; see also Section~\ref{sec:average}).  Then,
	as $N\to\infty$,
	\begin{equation}\label{eq:mad-log}
		\E\,|X-Np|\;\sim\;\sqrt{\frac{2Npq}{\pi}}\;
		\exp\Bigl(\sum_{k\ge1}\frac{a_k(h_N;p)}{N^k}\Bigr),
	\end{equation}
	where the coefficients contain \emph{no} elementary part:
	\begin{equation}\label{eq:ak}
		a_k(h;p)=\frac{(-1)^{k+1}}{k(k+1)}
		\Bigl[B_{k+1}-\bigl(p^{-k}+(-1)^{k+1}q^{-k}\bigr)B_{k+1}(h)\Bigr].
	\end{equation}
	Separating the parities, $a_k$ takes the form
	\begin{align}
		a_{2j-1}(h;p)&=\frac{1}{2j(2j-1)}
			\Bigl[B_{2j}-\bigl(p^{1-2j}+q^{1-2j}\bigr)B_{2j}(h)\Bigr],
			\label{eq:ak-odd}\\[2pt]
		a_{2j}(h;p)&=\frac{p^{-2j}-q^{-2j}}{2j(2j+1)}\;B_{2j+1}(h).
			\label{eq:ak-even}
	\end{align}
	Equivalently,
	\begin{equation}\label{eq:mad-mult}
		\E\,|X-Np|\;\sim\;\sqrt{\frac{2Npq}{\pi}}\sum_{m\ge0}\frac{d_m(h_N;p)}{N^m},
		\qquad d_0=1,\quad
		d_m=\frac1m\sum_{k=1}^{m}k\,a_k(h;p)\,d_{m-k}.
	\end{equation}
	More precisely, for every $M$ and every compact $K\subset(0,1)$,
	\begin{equation}\label{eq:mad-trunc}
		\E\,|X-Np|=\sqrt{\frac{2Npq}{\pi}}\,
		\Bigl(\sum_{m=0}^{M}\frac{d_m(h_N;p)}{N^{m}}+O(N^{-M-1})\Bigr),
	\end{equation}
	with an implied constant depending only on $M$ and $K$, uniformly for $p\in K$.  The remainder
	in \eqref{eq:mad-trunc} is \emph{relative} to the leading factor $\sqrt{2Npq/\pi}$; in
	absolute terms it is $O(N^{-M-\frac12})$.  In particular the expansion contains only integer
	powers of $N^{-1}$ relative to the leading term; no half-integer orders occur at any depth.
\end{theorem}

\begin{proof}
	By \eqref{eq:mad-as-local} and Theorem~\ref{thm:local},
	\[
		\log\E|X-Np|
		=\log\sqrt{\frac{2Npq}{\pi}}
		 +\log\Bigl(1+\frac{h_N}{Np}\Bigr)
		 +\sum_{k\ge1}\frac{A_k(h_N;p)}{N^k}.
	\]
	Insert \eqref{eq:Ak} and \eqref{eq:log-prefactor}: the term $(-1)^{k}h^k/(kp^k)$ of $A_k$ and
	the term $(-1)^{k+1}h^k/(kp^k)$ of the logarithm cancel, for every $k\ge1$, leaving
	\eqref{eq:ak}.  For \eqref{eq:ak-odd}--\eqref{eq:ak-even}, note that for even $k=2j$ we have
	$B_{k+1}=B_{2j+1}=0$ and $(-1)^{k+1}=-1$, while for odd $k=2j-1$ we have $(-1)^{k+1}=+1$.
	The multiplicative form and the recursion in \eqref{eq:mad-mult} follow as before.

	For the uniform statement \eqref{eq:mad-trunc}: apply Theorem~\ref{thm:local} with $H=[0,1]$,
	which is legitimate because $h_N\in[0,1)$ and $Np+h_N=\lceil Np\rceil$ for every $N$.
	The Taylor expansion \eqref{eq:log-prefactor} of $\log(1+h/(Np))$ also holds with a uniform $O(N^{-M-1})$
	remainder for $h\in[0,1]$ and $p\in K$, since then $h/(Np)\le 1/(N\min K)\to0$ uniformly.
	Adding the two truncated logarithms and exponentiating, as in
	Theorem~\ref{thm:multiplicative}, gives \eqref{eq:mad-trunc}.  Uniformity is what makes the
	statement meaningful at all: $h_N$ does not converge, so an expansion whose coefficients
	depend on $h_N$ is legitimate only if the error bounds do not.
\end{proof}

The first three \emph{logarithmic} coefficients are much simpler than the multiplicative ones,
and it is worth seeing them before the algebra of exponentiation begins:
\begin{equation}\label{eq:a123}
	a_1=\frac1{12}-\frac{B_2(h)}{2pq},
	\qquad
	a_2=\frac{q-p}{6\,p^2q^2}\,B_3(h),
	\qquad
	a_3=-\frac1{360}-\frac{1-3pq}{12\,p^3q^3}\,B_4(h),
\end{equation}
where we used $p^{-1}+q^{-1}=(pq)^{-1}$, $p^{-2}-q^{-2}=(q-p)(pq)^{-2}$ and
$p^{-3}+q^{-3}=(1-3pq)(pq)^{-3}$, all consequences of $p+q=1$.  The logarithm is where the
structure lives; exponentiation only mixes it.

\begin{remark}[The first coefficient is Frame's]\label{rem:frame}
	The coefficient $a_1$ of \eqref{eq:a123} is not new: it is the correction term of
	Frame~\cite{frame1945} and Johnson~\cite{johnson1957}, in disguise.  Solving Frame's
	\eqref{eq:frame} for $\E|X-Np|$ gives
	\[
		\E|X-Np|=\sqrt{\frac{2Npq}{\pi}}
		\Bigl(1+\frac{1}{N}\cdot\frac{h(1-h)-\tfrac16(1-pq)}{2pq}+O(N^{-2})\Bigr),
	\]
	and since $B_2(h)=h^2-h+\tfrac16$, that is
	\[
		\frac{h(1-h)-\tfrac16(1-pq)}{2pq}
		=\frac{\tfrac16-B_2(h)-\tfrac16+\tfrac16pq}{2pq}
		=\frac1{12}-\frac{B_2(h)}{2pq}
		=a_1(h;p).
	\]
	Johnson's form is the same statement linearised, his $\xi$ --- defined by the requirement that
	$Np+\xi$ be the least integer exceeding $Np$ --- being our $h_N$.

	Two remarks are in order.  First, this is a check: two independent derivations, seventy years
	apart, of the same coefficient.  Second, and more to the point, Frame writes the lattice
	dependence as the product $(Np-[Np])(Nq-[Nq])=h(1-h)$ of the two fractional parts, and there it
	stops --- the next order in \eqref{eq:frame} is bounded, not computed, and neither he nor
	Johnson has a structure to iterate.  The identity $h(1-h)=\tfrac16-B_2(h)$ is what turns their
	isolated correction into the first term of a sequence.  What Theorem~\ref{thm:mad} adds is
	precisely that: the recognition that the oscillating coefficients are Bernoulli polynomials of
	$h_N$, and hence all of them at once.
\end{remark}

\begin{remark}[Symmetry]\label{rem:symmetry}
	Relabelling the two outcomes replaces $X$ by $N-X\sim\Bin(N,q)$ and leaves $|X-Np|$
	unchanged, while $h_N\mapsto 1-h_N$ (for $h_N\in(0,1)$).  Accordingly the coefficients
	\eqref{eq:ak} are invariant under $(p,h)\mapsto(q,1-h)$: for odd $k$ the weight
	$p^{-k}+q^{-k}$ is symmetric and $B_{k+1}(1-h)=B_{k+1}(h)$, while for even $k$ the weight
	$p^{-k}-q^{-k}$ changes sign together with $B_{k+1}(1-h)=-B_{k+1}(h)$.  In the same vein,
	every coefficient $d_m$ below is a polynomial in $(pq)^{-1}$ and $q-p$, subject to
	$(q-p)^2=1-4pq$; the odd Bernoulli polynomials always appear multiplied by the
	antisymmetric factor $q-p$, and the even ones by symmetric factors.
\end{remark}

The first three multiplicative coefficients are
\begin{align}
	d_1&=\frac1{12}-\frac{B_2(h)}{2pq},
	\label{eq:d1}\\[4pt]
	d_2&=\frac{1+5p^2q^2}{1440\,p^2q^2}
	  +\frac{1-pq}{24\,p^2q^2}\,B_2(h)
	  +\frac{q-p}{6\,p^2q^2}\,B_3(h)
	  +\frac{1}{8\,p^2q^2}\,B_4(h),
	\label{eq:d2}\\[4pt]
	d_3&=\frac{1}{17280\,p^2q^2}-\frac{1}{181440\,p^3q^3}-\frac{973}{362880}
	\notag\\
	  &\quad
	  -\frac{(1-pq)^2}{576\,p^3q^3}\,B_2(h)
	  -\frac{(q-p)(3-pq)}{72\,p^3q^3}\,B_3(h)
	\notag\\
	  &\quad
	  -\frac{5(2-5pq)}{96\,p^3q^3}\,B_4(h)
	  -\frac{q-p}{12\,p^3q^3}\,B_5(h)
	  -\frac{1}{48\,p^3q^3}\,B_6(h),
	\label{eq:d3}
\end{align}
all evaluated at $h=h_N$.  Note that $1+5p^2q^2=5p^4-10p^3+5p^2+1$; the form on the left makes
the dependence on the single symmetric invariant $pq$ visible.

\begin{example}[The symmetric case]\label{ex:half}
	For $p=\tfrac12$ the displacement takes only two values: $h_N=0$ for even $N$ and
	$h_N=\tfrac12$ for odd $N$.  Since $B_2(0)=\tfrac16$ and $B_2(\tfrac12)=-\tfrac1{12}$,
	\eqref{eq:d1} gives $d_1=-\tfrac14$ for even $N$ and $d_1=+\tfrac14$ for odd $N$: the
	first correction is a pure parity oscillation.  At the next order the oscillation cancels:
	$B_3(0)=B_3(\tfrac12)=0$, $B_4(0)=-\tfrac1{30}$, $B_4(\tfrac12)=\tfrac7{240}$, and
	\eqref{eq:d2} gives $d_2=\tfrac1{32}$ for \emph{both} parities.  Thus
	\begin{align*}
		\E\,\Bigl|X-\frac N2\Bigr|
		&=\sqrt{\frac{N}{2\pi}}
		\Bigl[\,1\mp\frac1{4N}+\frac1{32N^2}+O(N^{-3})\Bigr],\\
		&\hspace{35mm} X\sim\Bin\bigl(N,\tfrac12\bigr),
	\end{align*}
	the sign being $-$ for even and $+$ for odd $N$.

	This particular case carries no news.  For even $N=2m$, De~Moivre's formula
	\eqref{eq:demoivre} reduces $\E|X-N/2|$ to a multiple of the central binomial coefficient
	$\binom{2m}m$, and the display above is then nothing but the classical
	\[
		\binom{2m}m=\frac{4^m}{\sqrt{\pi m}}
		\Bigl(1-\frac1{8m}+\frac1{128m^2}+O(m^{-3})\Bigr)
	\]
	rewritten in the variable $N=2m$; see for instance \cite{elezovic_jis} and the references
	there.  The oscillation between the two parities is likewise visible in \eqref{eq:frame}, since
	$h(1-h)$ takes the two values $0$ and $\tfrac14$.  We record the example because it is the one
	place where the general answer can be compared against a completely classical one --- and it
	agrees.  The content of Theorem~\ref{thm:mad} is the behaviour at a general $p$, where $h_N$ no
	longer takes finitely many values.
\end{example}

% =====================================================================
\section{Two-sided bounds}\label{sec:bounds}

An asymptotic expansion need not bound the function it represents, and the expansion of
Theorem~\ref{thm:mad} is a case in point.  Off the lattice the sign of its remainder is not
fixed: the coefficient $d_2(h;p)$ of \eqref{eq:d2} is positive for some displacements $h$ and
negative for others --- it carries the antisymmetric $B_3(h)$ and the sign-changing $B_4(h)$ ---
and likewise $d_3$.  Consequently, for a general $p$ the truncated partial sums of
\eqref{eq:mad-mult} do \emph{not} bracket $\E|X-Np|$; which side they fall on depends on $h_N$,
and $h_N$ oscillates.  This is the structural reason the classical literature turns to bounds
\emph{uniform in the lattice position} (Berend and Kontorovich~\cite{berend_kontorovich}) rather
than to truncations of an expansion.  We record two things the expansion nevertheless gives: an
unconditional two-sided bound of logarithmic width $O(N^{-2})$ in the interior, and, at an integer
mean, a collapse to a sign-alternating series in odd powers of $N^{-1}$ that is genuinely
enveloping (Theorem~\ref{thm:enveloping}).

\subsection*{An unconditional two-sided bound}

The exact reduction \eqref{eq:demoivre} turns any two-sided bound on a single factorial into one on
$\E|X-Np|$, in the non-degenerate range $1\le\nu\le N-1$.  We use Robbins's sharpening of Stirling's
formula~\cite{robbins},
\begin{equation}\label{eq:robbins}
	\sqrt{2\pi n}\,\Bigl(\tfrac ne\Bigr)^{n}e^{1/(12n+1)}
	<n!<
	\sqrt{2\pi n}\,\Bigl(\tfrac ne\Bigr)^{n}e^{1/(12n)},\qquad n\ge1 .
\end{equation}

\begin{proposition}[A rigorous two-sided bound]\label{prop:robbins}
	Let $X\sim\Bin(N,p)$, $q=1-p$, $\nu=\lceil Np\rceil$, and suppose $1\le\nu\le N-1$.  Then
	\begin{equation}\label{eq:robbins-bound}
		2\nu q\,e^{\,M_N+\varepsilon_N^{-}}\;\le\;\E|X-Np|\;\le\;2\nu q\,e^{\,M_N+\varepsilon_N^{+}},
	\end{equation}
	where
	\[
		M_N=\tfrac12\log\frac{N}{2\pi\nu(N-\nu)}
		 +\nu\log\frac{Np}{\nu}+(N-\nu)\log\frac{Nq}{N-\nu},
	\]
	\[
		\varepsilon_N^{+}=\frac1{12N}-\frac1{12\nu+1}-\frac1{12(N-\nu)+1},\qquad
		\varepsilon_N^{-}=\frac1{12N+1}-\frac1{12\nu}-\frac1{12(N-\nu)} .
	\]
	Both bounds are elementary, and their logarithmic gap is
	\begin{equation}\label{eq:robbins-gap}
		\varepsilon_N^{+}-\varepsilon_N^{-}
		=\sum_{n\in\{N,\,\nu,\,N-\nu\}}\Bigl(\frac1{12n}-\frac1{12n+1}\Bigr)
		<\frac1{144}\Bigl(\frac1{N^2}+\frac1{\nu^2}+\frac1{(N-\nu)^2}\Bigr).
	\end{equation}
\end{proposition}

\begin{proof}
	By Proposition~\ref{prop:demoivre}, $\E|X-Np|=2\nu q\,b(\nu;N,p)$ with
	$b(\nu;N,p)=\frac{N!}{\nu!\,(N-\nu)!}\,p^{\nu}q^{N-\nu}$.  Writing
	$n!=\sqrt{2\pi n}\,(n/e)^n e^{r_n}$ with $r_n\in\bigl(\tfrac1{12n+1},\tfrac1{12n}\bigr)$ by
	\eqref{eq:robbins}, we have $\log b(\nu;N,p)=M_N+r_N-r_\nu-r_{N-\nu}$.  Bounding $r_N$ above and
	$r_\nu,r_{N-\nu}$ below gives the upper bound in \eqref{eq:robbins-bound}, and the reverse choice
	gives the lower bound.  For \eqref{eq:robbins-gap} use
	$\frac1{12n}-\frac1{12n+1}=\frac1{12n(12n+1)}<\frac1{144\,n^2}$.
\end{proof}

\begin{remark}[What the second order buys]\label{rem:secondorder}
	The leading factor $\sqrt{2Npq/\pi}$ alone approximates $\E|X-Np|$ with relative error
	$d_1(h_N;p)/N+O(N^{-2})$, of order $N^{-1}$; the Frame--Johnson correction \eqref{eq:frame}
	removes that $N^{-1}$ term but is itself only an approximation.  The bracket
	\eqref{eq:robbins-bound} pins $\E|X-Np|$ from both sides with relative width
	\eqref{eq:robbins-gap} of order $N^{-2}$, uniformly in the lattice position and with no appeal to
	the (unbounded) tail of the expansion.  It is thus one order sharper than the leading-order
	bounds for all $N$ beyond a $p$-dependent threshold: at $p=\tfrac13,\ N=100$ the width
	\eqref{eq:robbins-gap} is $8\times10^{-6}$, against a leading-order relative error of
	$2\times10^{-3}$.  The non-asymptotic bounds of Berend and Kontorovich~\cite{berend_kontorovich}
	remain the sharper statement for the smallest $N$ and near the endpoints $p\to0,1$, where
	\eqref{eq:robbins-gap} degrades through the factors $\nu^{-2}$ and $(N-\nu)^{-2}$.
\end{remark}

\subsection*{Integer mean: a sign-alternating odd series}

At an integer mean the obstruction of the opening paragraph disappears.

\begin{theorem}[Parity collapse and sign alternation at $Np\in\Z$]\label{thm:envelope}
	Let $Np\in\Z$, so that $h_N=0$.  Then every even-order coefficient in the logarithmic expansion
	\eqref{eq:mad-log} vanishes, $a_2=a_4=\dots=0$, and
	\begin{equation}\label{eq:odd-series}
		\log\frac{\E|X-Np|}{\sqrt{2Npq/\pi}}\sim\sum_{j\ge1}\frac{a_{2j-1}(0;p)}{N^{2j-1}},
		\qquad
		a_{2j-1}(0;p)=\frac{B_{2j}}{(2j-1)\,2j}\Bigl[1-\bigl(p^{\,1-2j}+q^{\,1-2j}\bigr)\Bigr].
	\end{equation}
	These coefficients alternate strictly in sign:
	\begin{equation}\label{eq:alternation}
		(-1)^{j}\,a_{2j-1}(0;p)>0\qquad(j\ge1,\ 0<p<1).
	\end{equation}
\end{theorem}

\begin{proof}
	At $h=0$ the even-order coefficient $a_{2j}(0;p)$ of \eqref{eq:ak-even} carries the factor
	$B_{2j+1}(0)=B_{2j+1}=0$, so it vanishes; \eqref{eq:odd-series} is \eqref{eq:ak-odd} at $h=0$,
	using $B_{2j}(0)=B_{2j}$.  For \eqref{eq:alternation}, the bracket $1-(p^{1-2j}+q^{1-2j})$ is
	negative, since $p^{1-2j}+q^{1-2j}=p^{-(2j-1)}+q^{-(2j-1)}>2$ for $0<p<1$, while $B_{2j}$ has sign
	$(-1)^{j+1}$; the product has sign $(-1)^{j}$.
\end{proof}

The specialisation $h_N=0$ of Proposition~\ref{prop:robbins} is particularly clean: then $\nu=Np$,
$N-\nu=Nq$, and $M_N=\tfrac12\log\bigl(1/(2\pi Npq)\bigr)$, so $2\nu q\,e^{M_N}$ is exactly the
leading factor $\sqrt{2Npq/\pi}$.

\begin{corollary}[Rigorous bracket at an integer mean]\label{cor:integer-bracket}
	If $Np\in\Z$ and $1\le Np\le N-1$, then
	\begin{equation}\label{eq:integer-bracket}
		\sqrt{\frac{2Npq}{\pi}}\;e^{\,\varepsilon_N^{-}}\;\le\;\E|X-Np|\;\le\;
		\sqrt{\frac{2Npq}{\pi}}\;e^{\,\varepsilon_N^{+}},
		\qquad
		\varepsilon_N^{\pm}=\frac{a_1(0;p)}{N}+O\!\bigl(N^{-2}\bigr),
	\end{equation}
	with $\varepsilon_N^{\pm}$ the explicit quantities of Proposition~\ref{prop:robbins} and
	$a_1(0;p)=\tfrac1{12}\bigl(1-(pq)^{-1}\bigr)$.
\end{corollary}

\begin{proof}
	The bracket is \eqref{eq:robbins-bound} with $\nu=Np$, $N-\nu=Nq$ and
	$2\nu q\,e^{M_N}=\sqrt{2Npq/\pi}$ as just computed.  Expanding to first order,
	$\varepsilon_N^{\pm}=\tfrac1{12N}\bigl(1-p^{-1}-q^{-1}\bigr)+O(N^{-2})=a_1(0;p)/N+O(N^{-2})$,
	since $p^{-1}+q^{-1}=(pq)^{-1}$.
\end{proof}

	Sign alternation \eqref{eq:alternation} is the formal signature of an enveloping series.  The
	next theorem shows that the odd series \eqref{eq:odd-series} is in fact enveloping: the
	remainder after any number of terms has the sign of the first term omitted, so consecutive
	truncations bracket the logarithm.  The mechanism is a sign-definite integral representation of
	the combined three-gamma Stirling remainder.

\begin{theorem}[The odd series envelops]\label{thm:enveloping}
	Let $Np\in\Z$, $1\le Np\le N-1$, and write
	\[
		T_M:=\sum_{j=1}^{M}\frac{a_{2j-1}(0;p)}{N^{2j-1}},\qquad M\ge0,
	\]
	for the partial sums of \eqref{eq:odd-series}.  Then for every $M\ge0$
	\begin{equation}\label{eq:envelope-sign}
		(-1)^{M+1}\Bigl[\log\frac{\E|X-Np|}{\sqrt{2Npq/\pi}}-T_M\Bigr]\;>\;0 ,
	\end{equation}
	the sign of the first omitted term.  Consequently the logarithm lies strictly between any two
	consecutive partial sums; in particular, for all such $N$,
	\begin{equation}\label{eq:envelope-13}
		\sqrt{\frac{2Npq}{\pi}}\;e^{\,a_1(0;p)/N}
		\;<\;\E|X-Np|\;<\;
		\sqrt{\frac{2Npq}{\pi}}\;e^{\,a_1(0;p)/N+a_3(0;p)/N^{3}} .
	\end{equation}
\end{theorem}

\begin{proof}
	By Binet's first formula \cite[\S12.31]{ww},
	\[
		\log\Gamma(x)=\Bigl(x-\tfrac12\Bigr)\log x-x+\tfrac12\log2\pi+\mu(x),
	\]
	\[
		\mu(x)=\int_0^\infty\varphi(t)\,e^{-xt}\,dt,
		\quad
		\varphi(t)=\frac1t\Bigl(\frac1{e^{t}-1}-\frac1t+\frac12\Bigr),
	\]
	valid for $x>0$.  Applying this to the three factors of
	$\E|X-Np|=2Npq\,b(Np;N,p)$ and using the computation preceding
	Corollary~\ref{cor:integer-bracket} --- at $\nu=Np$ the elementary terms assemble exactly into
	the leading factor --- we obtain the \emph{exact} identity
	\begin{equation}\label{eq:binet-identity}
		\log\frac{\E|X-Np|}{\sqrt{2Npq/\pi}}
		=\mu(N)-\mu(Np)-\mu(Nq)
		=\int_0^\infty\varphi(t)\,\Delta_N(t)\,dt,
	\end{equation}
	where
	\[
		\Delta_N(t):=e^{-Nt}-e^{-Npt}-e^{-Nqt}.
	\]
	Two facts about the two factors of the integrand:

	(i) \emph{Kernel.}  By the Mittag--Leffler expansion of the hyperbolic cotangent
	\cite[\S7.4]{ww},
	\[
		\frac1{e^{t}-1}-\frac1t+\frac12=\sum_{k\ge1}\frac{2t}{t^{2}+4\pi^{2}k^{2}},
		\qquad\text{so}\qquad
		\varphi(t)=\sum_{k\ge1}\frac{2}{t^{2}+4\pi^{2}k^{2}} .
	\]
	Truncating the geometric series of each summand,
	\[
		\frac{2}{t^{2}+4\pi^{2}k^{2}}
		=\sum_{n=1}^{M}\frac{2(-1)^{n-1}t^{2n-2}}{(4\pi^{2}k^{2})^{n}}
		+\frac{2(-1)^{M}t^{2M}}{(4\pi^{2}k^{2})^{M}\,(t^{2}+4\pi^{2}k^{2})},
	\]
	and summing over $k$ with Euler's
	$\sum_{k\ge1}2(4\pi^{2}k^{2})^{-n}=2\zeta(2n)(2\pi)^{-2n}=(-1)^{n-1}B_{2n}/(2n)!$ gives, for
	every $M\ge0$ and $t>0$,
	\begin{equation}\label{eq:QM}
		Q_M(t):=\varphi(t)-\sum_{n=1}^{M}\frac{B_{2n}}{(2n)!}\,t^{2n-2}
		=2(-1)^{M}\sum_{k\ge1}\frac{t^{2M}}{(4\pi^{2}k^{2})^{M}\,(t^{2}+4\pi^{2}k^{2})},
	\end{equation}
	whence $(-1)^{M}Q_M(t)>0$ on $(0,\infty)$.

	(ii) \emph{Bracket.}  $\Delta_N(t)<0$ for every $t>0$: already $e^{-Npt}>e^{-Nt}$, since
	$p<1$.

	Finally, term-by-term integration with
	$\int_0^\infty t^{2n-2}e^{-xt}\,dt=(2n-2)!\,x^{1-2n}$ turns the subtracted polynomial of
	\eqref{eq:QM} into precisely the partial sum $T_M$:
	\[
		\int_0^\infty\frac{B_{2n}}{(2n)!}\,t^{2n-2}\,\Delta_N(t)\,dt
		=\frac{B_{2n}}{(2n-1)2n}
		\Bigl[N^{1-2n}-(Np)^{1-2n}-(Nq)^{1-2n}\Bigr]
		=\frac{a_{2n-1}(0;p)}{N^{2n-1}}
	\]
	by \eqref{eq:odd-series}.  Hence
	\[
		\log\frac{\E|X-Np|}{\sqrt{2Npq/\pi}}-T_M=\int_0^\infty Q_M(t)\,\Delta_N(t)\,dt ,
	\]
	an absolutely convergent integral ($Q_M(t)=O(t^{2M})$ at $0$ and $O(t^{2M-2}+t^{-1})$ at
	$\infty$, against the exponential decay of $\Delta_N$) whose integrand has, by (i) and (ii), the
	fixed sign $(-1)^{M+1}$.  This is \eqref{eq:envelope-sign}; since the signs at $M$ and $M+1$ are
	opposite, the value lies strictly between $T_M$ and $T_{M+1}$, and $M=1,2$ give
	\eqref{eq:envelope-13}.
\end{proof}

\begin{remark}\label{rem:enveloping}
	Theorem~\ref{thm:enveloping} and Corollary~\ref{cor:integer-bracket} are complementary: the
	former brackets $\E|X-Np|$ by truncations of its own expansion, with a width of the order of the
	first omitted term, the latter by Robbins-type perturbations of the leading term, with width
	$O(N^{-2})$ independent of the expansion.  The collapse to an odd series is the same phenomenon
	that governs the central binomial coefficient, whose expansion is likewise in odd powers about
	the centre~\cite{elezovic_jis}; Example~\ref{ex:half} is the instance $p=\tfrac12$.  For
	$Np\notin\Z$ no statement of this kind is possible --- the sign of $d_2(h_N;p)$ already depends
	on $h_N$, as noted at the head of this section.
\end{remark}

% =====================================================================
\section{Averaging the oscillation}\label{sec:average}

	The coefficients $d_m(h_N;p)$ of \eqref{eq:mad-mult} do not converge as $N\to\infty$.  It is
	natural to ask for their mean value, and the answer is supplied by the multiplication theorem for
	Bernoulli polynomials,
\begin{equation}\label{eq:mult-bernoulli}
		B_k(mt)=m^{k-1}\sum_{j=0}^{m-1}B_k\Bigl(t+\frac jm\Bigr),
		\qquad m\in\mathbb N .
\end{equation}

By construction each $d_m$ is a finite linear combination
\begin{equation}\label{eq:dm-structure}
	d_m(h;p)=\delta_m(p)+\sum_{j=2}^{2m}\gamma_{m,j}(p)\,B_j(h),
\end{equation}
with coefficients $\delta_m,\gamma_{m,j}\in\mathbb Q\bigl[(pq)^{-1},q-p\bigr]$; see
\eqref{eq:d1}--\eqref{eq:d3}.  We call $\delta_m(p)$ the \emph{smooth part} and the rest the
\emph{lattice part}.

Two facts underlie \eqref{eq:dm-structure}.  First, $\{B_j\}_{j\ge0}$ is a graded basis of the
polynomial ring, so any polynomial in $h$ can be written in it --- in particular a product such as
$B_2(h)^2=B_4(h)+\tfrac13B_2(h)+\tfrac1{180}$, of the kind produced by exponentiating.  Second,
and less obviously, \emph{no $B_1$-term occurs}.  Indeed, if $P=\sum_jc_jB_j$ then, since
$B_j(1)-B_j(0)=j\cdot0^{\,j-1}$ vanishes for $j\ge2$ and equals $1$ for $j=1$,
\begin{equation}\label{eq:c1-endpoints}
	c_1=P(1)-P(0).
\end{equation}
Every $a_k$ of \eqref{eq:ak} satisfies $a_k(0;p)=a_k(1;p)$, because $B_{k+1}(1)=B_{k+1}(0)$ for
$k\ge1$; hence so does every $d_m$, being a polynomial in the $a_k$; hence $c_1=0$ by
\eqref{eq:c1-endpoints}.

Both features of \eqref{eq:dm-structure} --- the coefficient ring and the top index $2m$ ---
hold in general, not only in the examples \eqref{eq:d1}--\eqref{eq:d3}.  By \eqref{eq:ak} each
$a_k$ is a $\mathbb Q\bigl[(pq)^{-1},q-p\bigr]$-multiple of the single polynomial $B_{k+1}(h)$,
the weight $p^{-k}+(-1)^{k+1}q^{-k}$ lying in that ring because $p+q=1$.  The recursion
\eqref{eq:mad-mult} writes $d_m$ as a $\mathbb Q$-combination of products
$a_{k_1}\cdots a_{k_r}$ with $k_1+\dots+k_r=m$, so its coefficients remain in
$\mathbb Q\bigl[(pq)^{-1},q-p\bigr]$ and its degree in $h$ is $\sum_i(k_i+1)=m+r\le2m$.  Expanded
in the Bernoulli basis and stripped of its $B_1$-term by the argument just given, this is
\eqref{eq:dm-structure}.

This is not an accident of the algebra.  The endpoints $h=0$ and $h=1$ describe the \emph{same}
situation, $Np\in\Z$, through the two admissible choices $\nu=Np$ and $\nu=Np+1$ in
\eqref{eq:demoivre}, and Proposition~\ref{prop:demoivre} says that these give the same value of
$\E|X-Np|$.  The absence of a $B_1$-term is precisely that identity, read off the expansion.

\begin{proposition}[Ces\`aro means]\label{prop:cesaro}
	Let $0<p<1$ and $h_N=\lceil Np\rceil-Np$.
	\begin{itemize}
	\item[(i)] If $p$ is irrational, then $(h_N)_{N\ge1}$ is equidistributed in $[0,1)$, and for
		every $m$
		\[
			\lim_{M\to\infty}\frac1M\sum_{N=1}^{M}d_m(h_N;p)=\delta_m(p).
		\]
	\item[(ii)] If $p=a/b$ with $\gcd(a,b)=1$, then $h_N$ is periodic in $N$ with period $b$ and
		takes each of the values $0,\tfrac1b,\dots,\tfrac{b-1}b$ exactly once per period.
		Consequently the mean of $d_m$ over a period is obtained from \eqref{eq:dm-structure}
		by the substitution
		\[
			B_j(h)\;\longmapsto\;b^{-j}B_j .
		\]
	\end{itemize}
\end{proposition}

\begin{proof}
	(i) By Weyl's theorem $(\{Np\})_N$ is equidistributed for irrational $p$, and
	$h_N=1-\{Np\}$; equidistribution is preserved.  Since $\int_0^1B_j(h)\,dh=0$ for every
	$j\ge1$, the lattice part of \eqref{eq:dm-structure} averages to zero.

	(ii) With $p=a/b$ and $\gcd(a,b)=1$, $Na\bmod b$ runs over all residues as $N$ runs over a
	period, and $h_N=\bigl((-Na)\bmod b\bigr)/b$; so each $j/b$ occurs once.  The mean of
	$B_j(h_N)$ over a period is $\frac1b\sum_{i=0}^{b-1}B_j(i/b)=b^{-j}B_j$, which is exactly
	\eqref{eq:mult-bernoulli}.
\end{proof}

\begin{example}
	Take $m=1$.  The smooth part of $d_1$ is $\delta_1=\tfrac1{12}$, so for irrational $p$ the
	first correction oscillates about $\tfrac1{12N}$.  For $p=a/b$ the mean of $d_1$ over a
	period is
	\[
		\frac1{12}-\frac{B_2}{2b^2pq}=\frac1{12}-\frac1{12\,b^2pq}
		=\frac{1}{12}\Bigl(1-\frac{1}{b^2pq}\Bigr).
	\]
	For $p=\tfrac12$ (so $b=2$, $pq=\tfrac14$) this vanishes, in agreement with
	Example~\ref{ex:half}, where $d_1=\mp\tfrac14$ has mean zero.  For $p=\tfrac13$ ($b=3$,
	$pq=\tfrac29$) the mean is $\tfrac1{12}(1-\tfrac12)=\tfrac1{24}$.
\end{example}

	Thus the same Bernoulli polynomials that carry the lattice oscillation also determine its mean:
	over a rational period, \(B_j(h)\) is replaced by \(b^{-j}B_j\).  The smooth part
	$\delta_m(p)$ is what a non-lattice calculation would produce; the lattice part is the correction
	that the kink of $|x-Np|$, sitting off the lattice, adds to it.

% =====================================================================
\section{Numerical illustration}\label{sec:numerics}

	We measure the accuracy by the number of exact decimal digits,
\[
	\mathrm{EDD}(N)=-\log_{10}\Bigl|1-\frac{\text{approximation}}{\E|X-Np|}\Bigr|,
\]
	the exact value being computed from \eqref{eq:demoivre}.  Column $(k)$ uses the leading term
	$\sqrt{2Npq/\pi}$ together with the first $k$ correction terms of \eqref{eq:mad-mult}.

\begin{table}[htbp]
\centering
\begin{tabular}{|r|c|c|rrrrrr|}
\hline
$N$ & $p$ & $h_N$ & $(0)$ & $(1)$ & $(2)$ & $(3)$ & $(4)$ & $(5)$\\
\hline
  50 & $2/3$ & $2/3$ & 2.4 & 4.6 & 6.6 & 7.9 & 9.5 & 10.6\\
 100 & $2/3$ & $1/3$ & 2.7 & 5.7 & 7.3 & 9.2 & 10.9 & 12.4\\
 500 & $2/3$ & $2/3$ & 3.4 & 6.6 & 9.6 & 11.9 & 14.6 & 16.6\\
1000 & $2/3$ & $1/3$ & 3.7 & 7.7 & 10.3 & 13.2 & 15.9 & 18.4\\
\hline
 100 & $0.5$ & $0$ & 2.6 & 5.5 & 7.4 & 10.0 & 11.3 & 13.8\\
1000 & $0.5$ & $0$ & 3.6 & 7.5 & 10.4 & 14.0 & 16.3 & 19.9\\
\hline
  50 & $1/\sqrt2$ & $0.645$ & 2.3 & 4.5 & 6.3 & 7.7 & 9.2 & 10.3\\
 100 & $1/\sqrt2$ & $0.289$ & 2.8 & 5.3 & 7.4 & 8.8 & 10.9 & 12.0\\
 500 & $1/\sqrt2$ & $0.447$ & 3.3 & 7.2 & 9.1 & 12.4 & 13.9 & 16.8\\
1000 & $1/\sqrt2$ & $0.893$ & 4.1 & 7.2 & 10.1 & 13.0 & 15.6 & 18.2\\
\hline
 100 & $0.1$ & $0$ & 2.1 & 4.4 & 5.6 & 7.5 & 8.1 & 9.9\\
 500 & $0.1$ & $0$ & 2.8 & 5.8 & 7.7 & 10.4 & 11.6 & 14.3\\
1000 & $0.1$ & $0$ & 3.1 & 6.4 & 8.6 & 11.6 & 13.1 & 16.1\\
\hline
\end{tabular}
\caption{Exact decimal digits in the expansion \eqref{eq:mad-mult} of $\E|X-Np|$.}
\label{tab:edd}
\end{table}

	Two features of Table~\ref{tab:edd} are worth naming.  First, each additional term gains
	roughly $\log_{10}N$ digits, which is the signature of a genuine expansion in integer powers of
	$N^{-1}$; there is no loss of half a power anywhere, in accordance with the last assertion of
	Theorem~\ref{thm:mad}.  Second, the row $p=2/3$ shows the periodic oscillation of $h_N$, while
	the irrational row $p=1/\sqrt2$ shows the non-periodic one; the numerical accuracy is stable in
	both cases because the error bounds are uniform in $h$.  The accuracy degrades slowly as $p$
	approaches the boundary of $(0,1)$ (the row $p=0.1$), which is visible in the factors
	$(pq)^{-m}$ in \eqref{eq:d1}--\eqref{eq:d3}.

% =====================================================================
\section{Concluding remarks}\label{sec:conclusion}

\begin{enumerate}
\item The result \eqref{eq:mad-log} is exact in the following sense: the expansion of
	$\E|X-Np|$ requires no probabilistic approximation at all.  De~Moivre's identity converts
	the expectation into one local mass, and the local mass is a gamma quotient.  Every
	subsequent step is Stirling's series and the Appell calculus of Bernoulli polynomials.  The
	Gaussian mean deviation $\sqrt{2Npq/\pi}$ appears at the end, as the leading term, rather
	than at the beginning, as an approximation.
\item The cancellation of the elementary part in Theorem~\ref{thm:mad} has a one-line mechanism.
	The prefactor $\nu=Np+h$ in \eqref{eq:demoivre} is the same $\nu$ that indexes the mass, so its
	logarithm $\log(1+h/(Np))$ contributes the series $\sum_k(-1)^{k+1}h^k/(k\,p^kN^k)$, and this is
	the exact negative of the non-Bernoulli residue $(-1)^kh^k/(k\,p^k)$ that the shift identity
	$B_{m}(h+1)=B_m(h)+mh^{m-1}$ leaves in the local mass \eqref{eq:Ak}; the two cancel term by term.
	It would still be worth knowing whether this is an instance of a general principle for Appell
	sequences --- a prefactor whose logarithm annihilates the difference-residue of the shift.
\item For $n\ge3$ bins the corresponding quantity --- the mean of the total-variation type
	statistic $\sum_j|X_j-Np_j|$ --- has no such closed form, and its expansion is a lattice
	Edgeworth expansion whose oscillating parts are Bernoulli polynomials evaluated at the
	fractional parts $\{Np_j\}$.  The present note is thus the exactly solvable end of that
	story, and it suggests which features of the general answer are structural (the Bernoulli
	polynomials, the parity of the weights) and which are artefacts of the approximation.  This
	last paragraph is motivation, not a claim proved here: the $n\ge3$ statements belong to a
	separate study, and everything asserted in the present note is contained in
	Theorems~\ref{thm:local}, \ref{thm:mad} and \ref{thm:envelope}, and
	Propositions~\ref{prop:robbins} and \ref{prop:cesaro}.
\end{enumerate}

% =====================================================================

\end{document}